\documentclass[12pt,a4paper]{amsart}
\usepackage{amssymb}
\usepackage{hyperref}
\usepackage[utf8]{inputenc}

\theoremstyle{plain}
\newtheorem{theorem}{Theorem}

\theoremstyle{definition}
\newtheorem{definition}[theorem]{Definition}
\newtheorem{example}[theorem]{Example}
\newtheorem{remark}[theorem]{Remark}

\newtheorem{conjecture}[theorem]{Conjecture}

\theoremstyle{remark}

\newcommand{\R}{\mathbb R}

\newcommand{\fa}{\mathfrak a}

\newcommand{\fg}{\mathfrak g}
\newcommand{\fh}{\mathfrak h}

\newcommand{\fk}{\mathfrak k}

\newcommand{\fp}{\mathfrak p}

\newcommand{\ft}{\mathfrak t}

\DeclareMathOperator{\SO}{SO}
\DeclareMathOperator{\SU}{SU}

\newcommand{\su}{\mathfrak{su}}

\DeclareMathOperator{\Span}{Span}
\newcommand{\inner}[2]{\langle {#1},{#2}\rangle }
\newcommand{\innerdots}{\langle {\cdot},{\cdot}\rangle }

\DeclareMathOperator{\Ad}{Ad}
\DeclareMathOperator{\ad}{ad}

\DeclareMathOperator{\diam}{diam}

\title[On the smallest Laplace eigenvalue]{On the smallest Laplace eigenvalue for naturally reductive metrics on compact simple Lie groups}

\author{Emilio A.~Lauret}
\address{Instituto de Matemática (INMABB), Departamento de Matemática, Universidad Nacional del Sur (UNS)-CONICET, Bahía Blanca, Argentina.}
\email{emilio.lauret@uns.edu.ar}

\thanks{This research was supported by grants from CONICET, FONCyT, SeCyT, and the Alexander von Humboldt Foundation (return fellowship)}
\date{December 2019}

\begin{document}
\maketitle

\begin{abstract}
Eldredge, Gordina and Saloff-Coste recently conjectured that, for a given compact connected Lie group $G$, there is a positive real number $C$ such that $\lambda_1(G,g)\operatorname{diam}(G,g)^2\leq C$ for all left-invariant metrics $g$ on $G$. 
In this short note, we establish the conjecture for the small subclass of naturally reductive left-invariant metrics on a compact simple Lie group. 
\end{abstract}

\bigskip 

For an arbitrary compact homogeneous Riemannian manifold $(M,g)$, Peter Li~\cite{Li80} proved that 
\begin{equation}\label{eq:Li}
\lambda_1(M,g)\geq \frac{\pi^2/4}{\diam(M,g)^2}. 
\end{equation}
Here, $\lambda_1(M,g)$ denotes the smallest positive eigenvalue of the Laplace--Beltrami operator on $(M,g)$ and $\diam(M,g)$ is the diameter of $(M,g)$, that is, the maximum Riemannian distance between two points in $M$. 
This lower bound has been recently improved by Judge and Lyons~\cite[Thm.~1.3]{JudgeLyons19}.

In contrast, there is no uniform upper bound for the term $\lambda_1(M,g)\,\diam(M,g)^2$ among all compact homogeneous Riemannian manifolds. 
For instance, the product $(M_n,g_n)$ of $n$ $d$-dimensional round spheres of constant curvature one satisfies $\lambda_1(M_n,g_n)=d$ and $\diam(M_n,g_n)= \sqrt{n}\pi$ goes to infinity when $n\to\infty $.

Eldredge, Gordina and Saloff-Coste have recently conjectured the existence of a uniform upper bound valid on special classes of homogeneous Riemannian manifolds, namely, the space of left-invariant metrics on a fixed compact connected Lie group.

\begin{conjecture}\cite[(1.2)]{EldredgeGordinaSaloff18} \label{conj}
Given $G$ a compact connected Lie group, there exists $C>0$ (depending only on $G$) such that 
\begin{equation}\label{eq:EGS}
\lambda_1(M,g)\leq \frac{C}{\diam(M,g)^2}
\end{equation}
for all left-invariant metrics $g$ on $G$. 
\end{conjecture}

Among many other results, they confirm its validity for $\SU(2)$ in  \cite[Thm.~8.5]{EldredgeGordinaSaloff18}.
Explicit values of $C$ for $\SU(2)$ and $\SO(3)$ can be found in \cite[Thm.~1.4]{Lauret-SpecSU(2)}. 

The main goal of this article is to give a simple and short proof of the validity of the weaker conjecture after restricting to naturally reductive left-invariant metrics on a compact connected simple Lie group $G$ (Theorem~\ref{thm:main} below).
The reader should not consider this result as a strong evidence of Conjecture~\ref{conj}. 

We next define naturally reductive metrics (see for instance \cite[\S7.G]{Besse}).
Let $(M,g)$ be a homogeneous Riemannian manifold. 
We fix a base point $m\in M$ and $H$ a transitive group of isometries of $(M,g)$.
Let $K$ be the isotropy subgroup at $m$, that is, $K=\{a\in H: a\cdot m=m\}$. 
The Lie algebra $\fh$ of $H$ decomposes into a sum $\fh = \fk \oplus \fp$, where $\fk$ denotes the Lie algebra of $K$ and $\fp$ is $\Ad(K)$-invariant. 
For $X\in\fh$, we write $X=X_\fk+X_\fp$ according to this decomposition.
We have the identifications $M \equiv H/K$ and $T_mM\equiv \fp$, and the metric $g$ on $M$ corresponds to an $\Ad(K)$-invariant inner product $\innerdots_m$ on $\fp$.

\begin{definition}\label{def:naturallyreductive}
A Riemannian manifold $(M,g)$ is said to be \emph{naturally reductive} if it admits a transitive action by isometries by a Lie group $H$ and an $\Ad(K)$-invariant complement $\fp$ as above such that 
\begin{equation}
\inner{[Z,X]_\fp}{Y}_m + \inner{X}{[Z,Y]_\fp}_m =0
\end{equation}
for all $X,Y,Z\in\fg$. (Here, $[\cdot,\cdot]$ denotes the bracket of the Lie algebra $\fh$.)
\end{definition}

A naturally reductive space can be seen as a generalization of a symmetric space.
Among their nice geometric properties, we have that every geodesic is an orbit of an one-parameter group of isometries. 
Normal homogeneous spaces are also naturally reductive.
However, the class of naturally reductive spaces
is much broader and contains many other interesting cases.

We now give a simple construction of naturally reductive metrics in the case of interest of this paper, that is, when $M$ is a simple compact connected Lie group. 

\begin{remark}\label{rem:construction}
Let $G$ be a semisimple compact connected Lie group.	
It is well known that the space of left-invariant metrics on $G$ is in 1-to-1 correspondence with the space of inner products on its Lie algebra $\fg$. 
Let $K$ be a closed subgroup of $G$ and let $\fp$ denote the orthogonal complement of $\fk$ in $\fg$ with respect to the Killing form $B_\fg$ of $\fg$. 
(We recall that $B_\fg$ is an $\Ad(G)$-invariant negative definite bilinear form on $\mathfrak g$.)
Given $h$ a bi-invariant metric on $K$ and $\alpha$ a positive real number, we define the left-invariant metric $g_{h,\alpha}$ on $G$ induced by the inner product on $\fg$ given by
\begin{equation}\label{eq:g_alpha,h}
g_{\alpha,h}(X,Y) = h(X_{\fk},Y_{\fk}) + \alpha \, (-B_\fg)(X_{\fp},Y_{\fp})
\qquad\text{ for }X,Y\in\fg.
\end{equation}

D'Atri and Ziller proved that $g_{\alpha,h}$ is naturally reductive for all $\alpha,h$ as above
(see \cite[Thm.~1]{DAtriZiller}).
(Note that the transitive group $H$ as in Definition~\ref{def:naturallyreductive} is $G\times K$ acting on $G$ as $(a,b)\cdot x = axb^{-1}$ for $a,x\in G$, $b\in K$.)
Moreover, they also proved that any naturally reductive metric on $G$ is isometric to one of the above form when $G$ is assumed simple (see \cite[Thm.~3]{DAtriZiller}).  
\end{remark}

We are now in position to prove the main theorem. 

\begin{theorem}\label{thm:main}
Let $G$ be a compact connected simple Lie group. 
There exists $C=C(G)>0$ such that 
\begin{equation}\label{eq:inequality}
\lambda_1(G,g) \leq \frac{C}{\diam(G,g)^2}
\end{equation}
for all naturally reductive metrics $g$ on $G$. 
\end{theorem}

\begin{proof}
Let $G$ be a compact connected simple Lie group. 
We pick any naturally reductive metric on $G$, that is, $g_{\alpha,h}$ as in \eqref{eq:g_alpha,h}, for some choice of $K$, $\alpha$, and $h$ as in Remark~\ref{rem:construction}. 
Since the term $\lambda_1(M,g)\diam(M,g)^2$ is invariant under positive scaling of $g$, we can assume without loosing generality that $\alpha=1$.
We abbreviate $g_h=g_{h,1}$.  
Moreover, we avoid the case when $\fp=0$ for being trivial since $g_{h}$ is necessarily a negative multiple of the Killing form $B_{\fg}$ of $\fg$.
Under this new assumption, $\fp\neq0$, we will not use that $h$ is a bi-invariant metric on $K$. 
More precisely, $g_h$ is defined as in \eqref{eq:g_alpha,h} with $\alpha=1$ and $h$ any inner product on $\mathfrak k$. 

We need to introduce some notions to give an upper bound for the diameter. 
Recall that a sub-Riemannian manifold is a triple $(M,\mathcal H,s)$, where $M$ is a smooth manifold, $\mathcal H$ is a subbundle of $TM$ and $s=(s_m)_{m\in M}$ denotes a family of inner product on $\mathcal H$ which smoothly vary with the base point
(see \cite{Montgomery-tour} for a general reference).
A smooth curve $\gamma$ on $(M,\mathcal H,g)$ is called \emph{horizontal} if $\gamma'(t)\in \mathcal H_{\gamma(t)}$ for all $t$. 
When $\mathcal H$ satisfies the \emph{bracket-generating condition} (i.e.\ the Lie algebra generated by vector fields in $\mathcal H$ spans at every point the tangent space of $M$), also known as \emph{Hörmander condition}, the Chow--Rashevskii Theorem ensures that $\diam (M,\mathcal H,s)<\infty$ when $M$ is compact, that is, any two points in $M$ can be joined by a horizontal curve on $(M,\mathcal H,s)$. 
It follows immediately that, if $g$ is a Riemannian metric on $M$ and we define the sub-Riemannian metric $s$ on $(M,\mathcal H)$ given by $s_m=g_m|_{\mathcal H_m}$ for all $m\in M$, then 
\begin{equation}\label{eq:cota-diam-gral}
\diam(M,g)\leq \diam(M,\mathcal H,s). 
\end{equation}

We consider the sub-Riemannian manifold $(G,\mathcal H,s)$ determined by 
\begin{align}
\mathcal H&=\bigcup_{a\in G}dL_a(\fp), &
s_a\big(dL_a(X),dL_a(Y)\big) &= -B_{\fg}(X,Y) ,
\end{align}
for $X,Y\in\fp$ and $a\in G$. 
Here, $L_a:G\to G$ is given by $L_a(x)=ax$ and $\fp$ is seen as a subspace of $T_eG\equiv \fg$. 
Sub-Riemannian structures on arbitrary Lie groups of this form are called left invariant.
Since $s$ is the restriction of $g_h$ to $\mathcal H$,  \eqref{eq:cota-diam-gral} gives 
\begin{equation}\label{eq:cota-diam_1}
\diam(G,g_h)\leq \diam (G,\mathcal H,s).
\end{equation} 

We next show that $\diam (G,\mathcal H,s)<\infty$.
It is sufficient to show that $\mathcal H$ satisfies the bracket-generating condition by Chow--Rashevskii Theorem. 
Since $\mathcal H$ is left invariant, this condition is equivalent to show that the Lie subalgebra of $\fg$ generated by $\fp$, say $\fa$, satisfies $\fa=\fg$.
We next prove that $\fa$ is an ideal of the simple Lie algebra $\fg$, giving the required assertion. 
For $X=X_{\fk}+X_{\fp}\in\fg$ and $Y\in\fa$, we have that $[X,Y] = [X_{\fk},Y]+ [X_{\fp},Y]$, so it remains to show that $[X_{\fk},Y]\in\fa$ since $[X_{\fp},Y]\in\fa$ is clear. 
The subspace $\fa$ is spanned by elements of the form $[Y_1,[Y_2,\cdots ,[Y_{n-1},Y_n]\cdots]]$ with $Y_1,\dots,Y_n\in\fp$. 
We next show by induction on $n$ that $[X_{\fk},Y]\in\fa$ for such element $Y$.
The case $n=1$ is clear since $[X_{\fk},Y]=[X_{\fk},Y_1]\in\fp$ because $Y=Y_1\in\fp$ and $\fp$ is $\ad(\fk)$-invariant. 
Furthermore, the Jacobi identity implies
\begin{equation}
[X_{\fk},Y] 
= \big[[X_{\fk},Y_1],[Y_2,\cdots ,[Y_{n-1},Y_n]\cdots]\big] + \big[Y_1,[X_{\fk},[Y_2,\cdots ,[Y_{n-1},Y_n]\cdots]]\big],
\end{equation}
thus the inductive step also follows. 
This concludes the proof of $\fa=\fg$ and its consequence $\diam (G,\mathcal H,s)<\infty$.

We now consider the Riemannian submersion with totally geodesic fibers given by 
\begin{equation}
(K,h)\longrightarrow (G,g_h)
\stackrel{\pi}{\longrightarrow }
(G/K,-B_\fg|_\fp),
\end{equation}
which is a particular case of the general construction in \cite[\S{}2.2]{Berard-BergeryBourguignon82} and \cite[Thm.~9.80]{Besse}.
The spectral theory of Riemannian submersions with totally geodesic fibers has been intensively studied; see for instance \cite{GilkeyLeahyPark-book}.
It is well known that if $f$ is an eigenfunction of the Laplace--Beltrami operator $\Delta_b$ of the base space $(G/K,-B_\fg|_\fp)$ with corresponding eigenvalue $\lambda$, then $f\circ \pi$ is an eigenfunction of the Laplace--Beltrami operator $\Delta_{g_h}$ of $(G,g_h)$ with corresponding eigenvalue $\lambda$. 
Consequently, the smallest positive eigenvalue of $\Delta_b$ is an upper bound for the smallest positive eigenvalue of $\Delta_{g_h}$, that is
\begin{equation}\label{eq:cota-lambda_1}
\lambda_1(G,g_h)\leq \lambda_1(G/K,-B_\fg|_\fp).
\end{equation}

We have shown that $\lambda_1(G,g_h) \diam(G,g_h)^2 \leq \lambda_1(G/K,-B_\fg|_\fp) \diam (G,\mathcal H,s)^2$ for every left-invariant metric $h$ on $\fk$, when $\fp\neq0$.
Note that the right hand side depends on $K$ and $\fp$, but not on $h$.
We conclude the proof by a finiteness argument on the choice of $K$ and $\fp$.

There is a finite collection $\mathcal K$ of closed subgroups of $G$ such that, for any naturally reductive left-invariant metric $g$ on $G$, there are $K\in\mathcal K$, $\alpha>0$, and $h$ a bi-invariant metric on $\fk$ such that $(G,g)$ is isometric to $(G,g_{h,\alpha})$ as in Remark~\ref{rem:construction} (see \cite[Cor.~3.7]{GordonSutton10}). 
Hence, by taking 
\begin{equation}
C=\max_{K\in\mathcal K\smallsetminus\{G\}} \; \left\{
\lambda_1(G/K,-B_\fg|_\fp) \;\diam (G,\fp,-B_\fg|_\fp)^2, \;
\lambda_1(G,-B_{\fg})\;\diam(G,-B_{\fg})^2 
\right\}, 
\end{equation}
we conclude that \eqref{eq:inequality} holds for all naturally reductive left-invariant metrics on $G$. 
\end{proof}

Inside the proof, for a fixed closed subgroup $K$ of $G$ with $\dim K<\dim G$ (i.e.\ $\fp\neq0$), it was not used that $h$ is a bi-invariant inner product on $\fk$. 
Thus, it was also proven the following statement. 

\begin{theorem}
Let $G$ be a compact connected simple Lie group, let $K$ be a closed subgroup of $G$ of dimension strictly less than $\dim G$, and let $\fp$ denote the orthogonal complement of $\fk$ in $\fg$ with respect to the Killing form $B_\fg$ of $\fg$. 
There exists $C=C(G,K)>0$ such that 
\begin{equation}\label{eq:inequality}
\lambda_1(G,g) \leq \frac{C}{\diam(G,g)^2}
\end{equation}
for every left-invariant metric $g$ on $G$ satisfying that $g(\fk,\fp)=0$ and $g|_{\fp}$ is a negative multiple of $B_\fg|_{\fg}$. 
\end{theorem}

Theorem~\ref{thm:main} applied to the well-known situation $G=\SU(2)$ returns that \eqref{eq:EGS} is valid for a codimension one subspace of the space of left-invariant metrics up to isometry, as shown in the next example. 
For higher-dimensional compact connected simple Lie groups, the analogous codimension increases considerably.

\begin{example}
We consider $G=\SU(2)$, which is diffeomorphic to the $3$-sphere $S^3$. 
The elements
\begin{align}\label{eq2:X1X2X3}
X_1&= \begin{pmatrix} i&0 \\ 0&-i \end{pmatrix}, &
X_2&= \begin{pmatrix} 0&1 \\ -1&0 \end{pmatrix}, &
X_3&= \begin{pmatrix} 0&i \\ i&0 \end{pmatrix},
\end{align}
form a basis of the Lie algebra $\su(2)$.
For positive real numbers $a_1,a_2,a_3$, let $g_{(a_1,a_2,a_3)}$ denote the left-invariant metric on $\SU(2)$ induced by the inner product on $\su(2)$ given by
$
g_{(a_1,a_2,a_3)}(X_i,X_j)= a_i^2\, \delta_{i,j}.
$
Note that $g_{(1,1,1)}$ is a negative multiple of $B_{\su(2)}$.
Although the dimension of the space of inner products on $\su(2)$ is $6$, Milnor~\cite{Milnor76} proved that every left-invariant metric on $\SU(2)$ is isometric to $g_{(a_1,a_2,a_3)}$ for some $a_1,a_2,a_3>0$.
(Permutations of $(a_1,a_2,a_3)$ do note change the isometry class of $g_{(a_1,a_2,a_3)}$; see e.g.\ \cite[Lem.~2.8]{EldredgeGordinaSaloff18}.) 

The only proper closed connected subgroup of $G$ up to conjugation is the torus 
\begin{equation}
T:=\left\{\begin{pmatrix}
e^{i\theta} \\ & e^{-i\theta}
\end{pmatrix} : \theta\in\R\right\}. 
\end{equation}
Up to homotheties, there is only one bi-invariant metric on $T$.
Since the Lie algebra of $T$ is $\ft:=\Span_\R\{X_1\}$, D'Atri and Ziller's Theorem mentioned in Remark~\ref{rem:construction} (\cite[Thm.~3]{DAtriZiller}) ensures that every naturally reductive left-invariant metric on $\SU(2)$ is isometric to 
\begin{equation}
g_{\alpha,\beta} := \beta\, g_{(1,1,1)}|_{\ft} \oplus  \alpha\, g_{(1,1,1)}|_{\fp} = g_{(\sqrt{\beta},\sqrt{\alpha},\sqrt{\alpha})}
\end{equation}
for some $\alpha,\beta>0$, where $\fp=\Span_\R\{X_2,X_3\}$.
These metrics are precisely the $3$-dimensional Berger spheres. 
\end{example}

\subsection*{Acknowledgments}
The author is greatly indebted to the referee for many accurate suggestions that have helped to improve significantly the presentation of the paper.

\bibliographystyle{plain}

\end{document}